\documentclass[11pt,article]{amsart}
\usepackage[all]{xy}
\usepackage{amssymb}
\usepackage{amsthm}
\usepackage{cite}
\usepackage{hyperref}
\usepackage[normalem]{ulem}
\usepackage{amsmath,mathtools}
\usepackage{amscd,enumitem}
\usepackage{caption}
\usepackage[justification=centering]{caption}
\usepackage{verbatim}
\usepackage{eurosym}
\usepackage{float}
\usepackage{color}
\usepackage{dcolumn}
\usepackage[mathscr]{eucal}
\usepackage[all]{xy}
\usepackage{bbm}
\usepackage{ytableau}
\usepackage[textheight=8.7in, textwidth=6.6in]{geometry}
\usepackage{tikz}
\newtheorem*{conj*}{Conjecture}
\newtheorem{theorem}{Theorem}[section]

\theoremstyle{definition}
\newtheorem*{remark}{Remark}

\theoremstyle{plain}

\newtheorem{lemma}[theorem]{Lemma}

\newcommand{\Q}{\mathbb{Q}}

\newcommand{\F}{\mathbb{F}}

\numberwithin{equation}{section}
\newtheoremstyle{example}
  {\topsep}   
  {\topsep}   
  {\normalfont}  
  {0pt}       
  {\bfseries} 
  {.}         
  {5pt plus 1pt minus 1pt} 
  {}          
\theoremstyle{example}

\def\({\left(}
\def\){\right)}

\def\lp{\left(}
\def\rp{\right)}

\def\lb{\left\{}
\def\rb{\right\}}
\def\lv{\left\lvert}
\def\rv{\right\rvert}

\usepackage{centernot}

\usepackage{thmtools}
\usepackage{thm-restate}

\usepackage{hyperref}

\usepackage{cleveref}

\usepackage{caption}

\DeclarePairedDelimiter{\floor}{\lfloor}{\rfloor}

\newcommand{\Addresses}{{
  \bigskip
  \footnotesize

  \textsc{Department of Mathematics, University of Utah,
    Salt Lake City, Utah 84112}\par\nopagebreak
  \textit{Email address}, A. Blaser: \texttt{avalon.blaser@utah.edu}

  \medskip

  \textsc{Department of Mathematics, University of Pennsylvania, Philadelphia, Pennsylvania 19104}\par\nopagebreak
  \textit{Email address}, M. Bradley: \texttt{mollykb@sas.upenn.edu}

  \medskip

  \textsc{Department of Mathematics, Harvey Mudd College,
    Claremont, California 91711}\par\nopagebreak
  \textit{Email address}, D. Vargas: \texttt{dvargas@g.hmc.edu}

  \medskip

  \textsc{Department of Mathematics and Statistics, Amherst College,
    Amherst, Massachusetts 01002}\par\nopagebreak
  \textit{Email address}, K. Xing: \texttt{kxing24@amherst.edu}

}}

\begin{document}
\title[Matrix Point Count Distributions on Some Algebraic Varieties]{Sato--Tate Type Distributions for Matrix Points on Elliptic Curves and Some K3 Surfaces}
\author{Avalon Blaser, Molly Bradley, Daniel Vargas, and Kathy Xing}

\maketitle

\begin{abstract}
Generalizing the problem of counting rational points on curves and surfaces over finite fields, we consider the setting of $n \times n$ matrix points with finite field entries. We obtain exact formulas for matrix point counts on elliptic curves and certain $K3$ surfaces for ``supersingular" primes. These exact formulas, which involve partitions of integers up to $n$, essentially coincide with the expected value for the number of such points. Therefore, in analogy with the Sato--Tate conjecture, it is natural to study the distribution of the deviation from the expected values for all primes. We determine the limiting distributions for elliptic curves and a family of $K3$ surfaces. For non-CM elliptic curves with square-free conductor, our results are explicit.

\end{abstract}

\section{Introduction and Statement of Results}

If $E/\Q$ is an elliptic curve, Hasse \cite{hasse1936} proved that for every prime $p$, the $\F_p$-point count is given by
\begin{equation}\label{eqn: Hasse bound}
    \#E(\F_p)=p+1 - a_E(p) ,
\end{equation}
where $\lv a_E(p)\rv < 2\sqrt{p}$. We refer to $a_E(p)$ as the trace of Frobenius for elliptic curves. The arithmetic statistics for this term's distribution have been widely studied. In the case where $p$ is supersingular for $E$, it is well-known that there is a simple, explicit formula for $\F_p$-point counts; if $p$ is supersingular, then $\#E(\F_p)=p+1$.

Motivated by questions in geometry and combinatorics, Huang \cite{huang2022mutually} examined the notion of ``matrix points" on varieties. We investigate the arithmetic aspect of counting such matrix points on certain varieties, including elliptic curves. To make this precise, let $C_{n, m}(\mathbb{F}_p)$ be the set of pairwise-commuting $m$-tuples of $n \times n$ matrices over $\mathbb{F}_p$. If $n\geq 1$, $p$ is a prime, and $E/\mathbb{F}_p$ is given by 
$$E : y^2 + \alpha_1 xy + \alpha_3 y = x^3 + \alpha_2 x^2 + \alpha_4 x + \alpha_6,$$ 
then we define the number of $n \times n$ matrix points on $E$ by
\begin{equation} \label{eq: matrix pts elliptic curve}
N_{n}(E, \mathbb{F}_p) \coloneqq \# \{ (A, B) \in C_{n, 2}(\mathbb{F}_p) \colon B^2 + \alpha_1 AB + \alpha_3 B= A^3 + \alpha_2 A^2 + \alpha_4 A + \alpha_6 I_n \} .
\end{equation}
Since the term $xy$ appears in the equation defining $E,$ it is natural that the matrix points are restricted to those in $C_{n,2}(\F_p)$. There are known formulas for $N_{n}(E, \mathbb{F}_p)$ in terms of $a_E(p)$, derived by Huang, Ono, and Saad \cite{huangonosaad}. These involve $q$-multinomial coefficients
\begin{equation*}
    \binom{n}{m_1, m_2, \cdots\!, m_k}_q \coloneqq \frac{(q; q)_n}{(q; q)_{m_1} (q; q)_{m_2} \cdots (q; q)_{m_k}}
    ,
\end{equation*}
where $(q;q)_0 \coloneqq 1$ and $(q; q)_n \coloneqq \prod_{i=1}^n (1-q^i)$ for $n \geq 1$. Leveraging these known formulas, we find the matrix point analogue to the classical explicit formula for $\F_p$-point counts when $p$ is supersingular.

\begin{theorem}\label{thm: supersingular formula}
    If $n \geq 1$, $E$/$\Q$ is an elliptic curve, and $p \geq 3$ is a supersingular prime for $E$, then
    $$ 
    N_{n}(E, \mathbb{F}_p) = p^{\frac{n^2}{2}} \sum_{\substack{r + s + u = n \\ r, s, u \ge 0 \\ r \:\! \equiv \:\! s \:\!\!\!\!\! \mod 2}} (-1)^{\frac{r-s}{2}} p^{\frac{u^2}{2}} \binom{n}{r, s, u}_p.
    $$ 
\end{theorem}

\begin{remark}
When $E/ \Q$ is an elliptic curve with CM, Theorem \ref{thm: supersingular formula} applies for the primes $p$ that are inert in the CM field, which occur with density 1/2 in the set of all primes \cite{hecke}. In the non-CM case, even though the supersingular primes have density 0, Elkies \cite{elkies1987} proved that there are infinitely many such primes. 
\end{remark}

Moreover, we consider matrix points on the $K3$ surfaces with function fields given by \begin{equation} \label{eq: function field}
X_\lambda:s^2 = xy(x + 1)(y + 1)(x + \lambda y)
\end{equation}
where $\lambda \in \mathbb{Q} \setminus \{0, -1 \}$. This family of surfaces was discovered by Ahlgren, Ono, and Penniston \cite{aopk3}, who proved that for $p \ge 3$, the $\F_p$-point count is given by
\begin{equation}
\# X_{\lambda} (\mathbb{F}_p) = 1 + p^2 + 19p + p \cdot A_{\lambda}^*(p),
\end{equation}
where $\lv A_{\lambda}^*(p)\rv\leq 3$. One of the interesting properties of this family is that the zeta function of these surfaces is essentially the symmetric square of the zeta function of the Clausen elliptic curve
\begin{equation} \label{eq: Clausen curve}
E_{- \frac{\lambda}{\lambda + 1}}^{\text{CL}} : y^2 = (x + 1) \lp x^2 - \frac{\lambda}{\lambda + 1} \rp.
\end{equation}
If $n \ge 1$, the number of $n \times n$ matrix points on $X_\lambda$ is given by
\begin{equation} \label{eq: matrix point k3 surface def}
N_{n}(X_{\lambda}, \mathbb{F}_p) \coloneqq \# \{ (A, B, C)\in C_{n,3}(\F_p) \colon  C^2=AB(A+I_n)(B+I_n)(A+\lambda B), C \in \mathit{GL}_n(\mathbb{F}_p)\} .
\end{equation}

Using combinatorial input from partitions, Huang, Ono, and Saad \cite{huangonosaad} determined a general formula for $N_{n}(X_{\lambda}, \mathbb{F}_p)$ in terms of the Frobenius traces for the associated Clausen curve. Utilizing this formula, we explicitly determine $N_{n}(X_{\lambda}, \mathbb{F}_p)$ when $p$ is a supersingular prime for the associated Clausen curve.
This involves summing over all possible six-tuples of partitions $\lp \lambda_1, \dots, \lambda_6 \rp$ whose total size does not exceed $n$. We denote the size of $\lambda_i$ by $\lv \lambda_i \rv$, the length of $\lambda_i$ by $l(\lambda_i)$, and the number of $j$-parts in  $\lambda_i$ by $n(\lambda_i, j)$. In this notation, we have

\begin{theorem} \label{thm: k3 explicit formula}

If $n \ge 1$, $\lambda \in \Q \setminus \{0, -1\} $, and  $p \geq 3$ is a supersingular prime for the associated Clausen elliptic curve, then 
{\small
\begin{equation*}
    N_{n}(X_{\lambda}, \mathbb{F}_p) 
    = \sum\limits_{r=0}^n b_{n-r} \!\! \sum\limits_{\substack{\lambda_1 ,\dots, \lambda_6 \\ \lv\lambda_1\rv + \dots + \lv\lambda_6\rv = r}} (-1)^{l(\lambda_1) + \dots + l(\lambda_4)} \frac{p^{\frac{\sum n(\lambda_i , j)(n(\lambda_i , j) - 1)}{2}} p^{2l(\lambda_3) + l(\lambda_4) + l(\lambda_5) + l(\lambda_6)}}{\prod (p;p)_{n(\lambda_i , j)}} \gamma^{l(\lambda_4) + l(\lambda_5) + l(\lambda_6)}
    ,
\end{equation*}}
where $\gamma \coloneqq \phi_p(\lambda + 1)$, with $\phi_p$ the Legendre character of  $\mathbb{F}_p,$ and where  $\sum\limits_{r = 0}^{\infty} b_rq^r \coloneqq \prod\limits_{i=1}^\infty \lp 1 - q^i \rp^5$.
\end{theorem}

In the absence of exact formulas for matrix point counts, we study the distribution of their normalized errors as $p$ ranges over all primes. This is analogous to the Sato--Tate conjecture that $a_E^*(p) \coloneqq a_E(p)/\sqrt{p}$ follows a semicircular distribution as $p$ varies. More precisely, the conjecture states that if $E$ does not have complex multiplication and $-2 \le a < b \le 2$, then
\begin{equation} \label{eq: og Sato--Tate conjecture}
    \lim_{X\to\infty} \frac{\#\{ p \leq X \mid a_{E}^*(p) \in \left[a,b\right] \}}{\#\{ p \leq X \}} = \frac{1}{2\pi} \int_a^b \sqrt{4 - t^2}\ dt
    .
\end{equation}
Almost fifty years after it was first proposed, this conjecture was proved in a series of landmark papers by Barnet-Lamb, Clozel, Geraghty, Harris, Shepherd-Barron, and Taylor \cite{sato1, sato2, taylor1, taylor2}.

As is the case for $\F_p$-point counts, it turns out that matrix point counts have an expected value and an error term. More precisely, the matrix point error term is given by
\begin{equation} \label{eqn: aen}
    a_{E,n}(p) \coloneqq P(n,0)_p - N_n(E, \F_p) = -\sum_{k=1}^{n}\frac{a_E(p^k)}{p^k}P(n,k)_p,
\end{equation}
where $P(n, k)_p$ is a polynomial in $p$; see \eqref{eqn: P(n,k)}. 

When $E$ is non-CM and $p$ varies, we prove that $a_{E,n}^*(p) \coloneqq a_{E,n}(p)/p^{n^2-\frac{1}{2}}$ has a semicircular distribution, thereby recognizing the Sato--Tate conjecture as a special case of an infinite family of analogous results. On the other hand, when $E$ has CM and $p$ varies over the primes that split in the CM field, we prove that $a_{E,n}^*(p)$ has the same distribution as the finite field analogue $a_E^*(p)$.

\begin{theorem} \label{thm: ineffective theorems}
The following are true for all $n \ge 1$ and $-2 \le a < b \le 2$.
\begin{enumerate}[label=(\roman*)]
    \item If $E$/$\mathbb{Q}$ is an elliptic curve without CM, then 
    \begin{equation*}
    \lim_{X \to \infty}  \frac{\#\{p \le X \mid a_{E, n}^*(p) \in [a, b]\}}{ \# \{ p\le X\}}  = \frac{1}{2 \pi} \int_a^b \sqrt{4 - t^2} \, dt.
    \end{equation*}

    \item If $E$/$\mathbb{Q}$ is a CM elliptic curve with CM field $K$, then
    \begin{equation*}
    \lim_{X \to \infty} \frac{\#\{p \le X \mid p \text{ splits in $K$}, a_{E, n}^*(p) \in [a, b]\}}{\#\{p \le X \mid p \text{ splits in $K$}\}} = \
    \frac{1}{\pi} \int_a^b \frac{1}{\sqrt{4-t^2}} \, dt
    .
    \end{equation*}
\end{enumerate}
\end{theorem}

When $E$ is non-CM with square-free conductor $N$, we make Theorem \ref{thm: ineffective theorems} explicit. This is a generalization of the case $n = 1$ obtained by Hoey, Iskander, Jin, and Suárez \cite{hoey2022unconditional}.

\begin{theorem} \label{thm: effective theorems}
Let $E/\mathbb{Q}$ be a non-CM elliptic curve with square-free conductor $N$. If $X \ge 289$ and $-2 \le a < b \le 2$, then 
\begin{equation*}
\lv \frac{\#\{p \le X \mid a_{E, n}^*(p) \in [a, b]\}}{\#\{p \le X\}} - \frac{1}{2\pi} \int_{a}^{b} \sqrt{4-t^2} \, dt \rv \le 
58.44 \frac{\log \lp N \log X \rp}{\sqrt{\log X}}
.
\end{equation*}
\end{theorem}

We also study the distributions of matrix point error terms for the $K3$ surfaces defined in \eqref{eq: function field}. Ono, Saad, and Saikia \cite{onosaadsaikia} pioneered the study of the statistics of $A_{\lambda}^*(p)$ by determining the limiting distribution as $\lambda$ varies. On the other hand, Chen, Shen, and Thorner \cite{chenshenthorner} determined the distributions of these quantities as $p$ varies and $\lambda$ is fixed. These distributions depend on whether the associated Clausen elliptic curve has complex multiplication and whether $\lambda + 1$ is a rational square. 

As in the elliptic curve case, $N_n(X_\lambda, \mathbb{F}_p)$ has an expected value and an error term. More precisely, the matrix point error term is given by
\begin{equation} \label{eqn: alambdan}
    A_{\lambda,n}(p) \coloneqq Q(n, 0, \phi_p(\lambda + 1)) + R(n, \phi_p(\lambda + 1))_p - N_{n}(X_\lambda, \F_p)
    ,
\end{equation}
where $Q(n, k, \gamma)_p$ and $R(n, \gamma)_p$ are polynomials in $p$ and $\gamma$ \cite[Theorem.\ 1.4]{huangonosaad}. We prove that when $p$ varies and $\lambda$ is fixed, the normalized matrix point error term $A_{\lambda,n}^*(p) \coloneqq A_{\lambda,n}(p) / p^{n^2+n-1}$ has the same distribution as the finite field analogue $A_\lambda^*(p)$.

\begin{theorem} \label{thm: K3 bounding theorems}
If $n \geq 1$ and $\lambda \in \mathbb{Q} \setminus \{0, -1\}$, then the following are true with $B_1(t), B_2(t), B_3(t)$, and $B_4(t)$ as defined in \eqref{eqn: B_1 definition} through \eqref{eqn: B_4 definition} and $K$ as defined in Table \ref{table: K3 CM fields}.
\begin{enumerate} [label=(\roman*)]
    \item If $\lambda \notin (\mathbb{Q}^2 - 1) \cup \{ 1/8, 1, -1/4, -1/64, -4, -64 \}$ and $-3 \leq a < b \leq 3$, then
    \begin{equation*}
        \lim_{X \to \infty} \frac{\# \{ p \leq X \mid A_{\lambda,n}^*(p) \in [a, b] \}}{\# \lb p \leq X \rb} = \int_a^b B_1(t) \, dt. 
    \end{equation*}

    \item If $\lambda \in (\mathbb{Q}^2 - 1) \setminus \{ 0, -1, 8 \}$ and $-1 \leq a < b \leq 3$, then 
    \begin{equation*}
        \lim_{X \to \infty} \frac{\# \{ p \le X \mid A_{\lambda,n}^*(p) \in [a, b]\}}{\# \{p \le X \}} = \int_a^b B_2(t) \, dt.
    \end{equation*}

    \item If $\lambda \in \{ 1/8, 1, -1/4, -1/64 \}$ and $-3 \leq a < b \leq 3$, then
    \begin{equation*}
        \lim_{X \to \infty} \frac{\# \{ p \leq X \mid p \text{ splits in $K$}, A_{\lambda, n}^*(p) \in [a,b] \}}{\# \{ p \leq X \mid p \text{ splits in $K$} \}} = \int_{a}^{b} B_3(t) \, dt.
    \end{equation*}

    \item If $\lambda \in \{-4, 8 -64 \}$ and $-1 \leq a < b \leq 3$, then 
    \begin{equation*}
        \lim_{X \to \infty} \frac{\# \{ p \le X \mid p \text{ splits in $K$}, A_{\lambda,n}^*(p) \in [a, b]\}}{\# \{p \le X \mid p \text{ splits in $K$} \}} = \int_a^b B_4(t) \, dt.
    \end{equation*}
\end{enumerate}
\end{theorem}

This paper is organized as follows. In Section \ref{section: 2}, we outline some key facts about elliptic curves, the $K3$ surfaces $X_\lambda$, and matrix points. In Section \ref{section: 3}, we derive some explicit bounds needed in the proof of Theorem \ref{thm: effective theorems}. In Section \ref{section: 4},  we prove Theorems \ref{thm: supersingular formula}, \ref{thm: k3 explicit formula}, \ref{thm: ineffective theorems}, \ref{thm: effective theorems}, and \ref{thm: K3 bounding theorems}. Finally, in Section \ref{section: 5}, we illustrate Theorems \ref{thm: ineffective theorems} and \ref{thm: K3 bounding theorems} with numerical examples. 

\section*{Acknowledgements}

The authors would like to thank Yifeng Huang and Hasan Saad for advising and supporting this project, as well as Eleanor McSpirit for her helpful comments and suggestions. The authors would also like to thank Hasan Saad for producing the histograms for Section 5. The authors were participants in the 2023 UVA REU in Number Theory. They would like to thank Ken Ono for running this program and for his guidance and advice. They are grateful for the support of grants from Jane Street Capital, the National Science Foundation (DMS-2002265 and DMS-2147273), the National Security Agency (H98230-23-1-0016), and the Templeton World Charity Foundation. The authors used Wolfram Mathematica to simplify computations.

\section{Arithmetic Geometric Preliminaries} \label{section: 2}

In order to obtain our results, we exploit information about $\F_p$-point counts and their connection to matrix point counts. To this end, this section recalls key facts about elliptic curves, $K3$ surfaces, and matrix points.

\subsection{Elliptic Curves} \label{sec: 2.1} 
We collect some results concerning distributions of $\F_p$-point counts on elliptic curves. 

If $E$ has CM, then $a_E^*(p) = 0$ for all primes inert in the CM field $K$. On the other hand, for primes that split in $K$, work of Hecke \cite{hecke} and Deuring \cite{silvermanbook} implies the following distribution. Namely, if $-2 \le a < b \le 2$, then
\begin{equation} \label{eqn: cm limiting distribution}
\lim_{X \to \infty} \frac{\# \{ p \leq X \mid p \text{ splits in } K, a_E^*(p) \in [a, b] \}}{\# \{ p \leq X \mid p \text{ splits in } K \}} = \frac{1}{\pi} \int_a^b \frac{1}{\sqrt{4-t^2}} \, dt . 
\end{equation}

If $E$ does not have CM, we recall that $a_E^*(p)$ has a semicircular distribution \eqref{eq: og Sato--Tate conjecture}. In Theorem \ref{thm: effective theorems}, we explicitly bound the difference between the distribution of $a_{E,n}^*(p)$ for some elliptic curves and the semicircular distribution. To do this, we use the explicit bound for the case $n = 1$.

\begin{theorem}[Th.\ 1.1 of \cite{hoey2022unconditional}]
If $E/\Q$ is a non-CM elliptic curve with square-free conductor $N$, $X \geq 3$, and $-2 \le a < b \le 2$, then
\begin{equation} \label{eqn: ST bound}
    \lv \frac{\# \{ p \leq X \mid a_E^*(p) \in [a, b] \}}{\# \{ p \leq X \}} - \frac{1}{2\pi} \int_{a}^{b} \sqrt{4 - t^2} \, dt \rv \leq 58.1\frac{\log \lp N \log X \rp}{\sqrt{\log X}} 
    .
\end{equation}
\end{theorem}

\begin{remark}
    In some of our proofs, it will be convenient to denote the Sato--Tate distribution by 
    \begin{equation}
        f_{ST}(t) \coloneqq
        \begin{cases}
            \frac{1}{2\pi} \sqrt{4 - t^2} & \text{if } -2 \leq t \leq 2\\
            0 & \text{else}.
        \end{cases}
    \end{equation}
\end{remark}

\subsection{K3 Surfaces} \label{sec: 2.2}

As previously discussed, each $K3$ surface $X_\lambda$ has an associated Clausen elliptic curve $E_{-\lambda/(\lambda + 1)}^{\text{CL}}$. Results on this family of $K3$ surfaces are dependent on whether $\lambda+1$ is a rational square and whether $E_{-\lambda/(\lambda + 1)}^{\text{CL}}$ has CM. When $E_{-\lambda/(\lambda + 1)}^{\text{CL}}$ has CM, our distributions depend on its CM field $K$. The CM fields for all $\lambda$ where this applies are listed in Table \ref{table: K3 CM fields}.

\bgroup
\def\arraystretch{1.75}%
\begin{table}[H]
\begin{center}
\caption{\label{table: K3 CM fields}$\lambda$ and $K$ when $E_{-\lambda/(\lambda + 1)}^{\text{CL}}$ has CM}
\begin{tabular}{|p{3.5em}|p{3.5em}|p{3.5em}|p{3.5em}|p{3.5em}|p{3.5em}|p{3.5em}|p{3.5em}|} 
 \hline
 $\lambda$ & $8$ & $1/8$ & $1$ & $-4$ & $-1/4$ & $-64$ & $-1/64$ \\ 
 \hline
 $K$ & $\Q(\sqrt{-1})$ & $\Q(\sqrt{-1})$ & $\Q(\sqrt{-2})$ & $\Q(\sqrt{-3})$ & $\Q(\sqrt{-3})$ & $\Q(\sqrt{-7})$ & $\Q(\sqrt{-7})$ \\
 \hline
\end{tabular}
\end{center}
\end{table}
\egroup

We now define the probability density functions\footnote{Each of these density functions is named as a variation of Batman because of its shape. Histograms overlaid with these density functions are given in Section \ref{section: 5}.} used in Theorem \ref{thm: K3 bounding theorems}:
\begin{align}
    B_1(t) &\coloneqq 
    \begin{cases}
        \frac{3 + t}{4\pi\sqrt{3-2t-t^2}} + \frac{3 - t}{4\pi\sqrt{3+2t-t^2}} &\quad \lv t \rv < 1
        \\ \frac{3 - \lv t \rv }{4\pi\sqrt{3+2\lv t \rv -t^2}} &\quad 1 \leq \lv t \rv \leq 3
        \\ 0 &\quad \text{otherwise}
    \end{cases} & \text{(original Batman)}
    \label{eqn: B_1 definition}
    \\ B_2(t) &\coloneqq 
    \begin{cases} 
        \frac{1}{2\pi}\sqrt{\frac{3-t}{1+t}} &\hspace{88pt} -1 < t \leq 3
        \\ 0 &\hspace{88pt} \text{otherwise}
    \end{cases} & \text{(half-Batman)}
    \label{eqn: B_2 definition}
    \\ B_3(t) &\coloneqq 
    \begin{cases}
        \frac{1}{2\pi \sqrt{3-2t-t^2}} + \frac{1}{2\pi \sqrt{3+2t-t^2}} &\quad \lv t \rv < 1
        \\ \frac{1}{2\pi \sqrt{3+2\lv t \rv -t^2}} &\quad 1 \leq \lv t \rv \leq 3
        \\ 0 &\quad \text{otherwise}
    \end{cases} & \text{(flying Batman)}
    \label{eqn: B_3 definition}
    \\ B_4(t) &\coloneqq 
    \begin{cases} 
        \frac{1}{\pi \sqrt{3+2t-t^2}} & \hspace{80pt} -1 < t < 3
        \\ 0 & \hspace{80pt} \text{otherwise}
    \end{cases} & \text{(half-flying Batman)}
    \label{eqn: B_4 definition}
\end{align}

In order to determine the distribution of $A_{\lambda, n}^*(p)$, we use the following result concerning the distribution of $A_\lambda^*(p)$.

\begin{theorem}[Th.\ 2.1 of \cite{chenshenthorner}] \label{thm: 2.2}
The following are true.
\begin{enumerate} [label=(\roman*)]
    \item If $\lambda \notin (\mathbb{Q}^2 - 1) \cup \{ 1/8, 1, -1/4, -1/64, -4, -64 \}$ and $-3 \leq a < b \leq 3$, then 
    \begin{equation} \label{B1}
        \lim_{X \to \infty} \frac{\# \{ p \le X \mid A^*_{\lambda}(p) \in [a, b]\}}{\# \{p \le X \}} = \int_a^b B_1(t) \, dt
        .
    \end{equation} 

    \item If $\lambda \in (\mathbb{Q}^2 - 1) \setminus \{ 0, -1, 8 \}$ and $-1 \leq a < b \leq 3$, then
    \begin{equation} \label{B2}
        \lim_{X \to \infty} \frac{\# \{ p \le X \mid A^*_{\lambda}(p) \in [a, b]\}}{\# \{p \le X \}} = \int_a^b B_2(t) \, dt
        .
    \end{equation}
    
    \item If $\lambda \in \{ 1/8, 1, -1/4, -1/64 \}$ and $-3 \leq a < b \leq 3$, then 
    \begin{equation} \label{B3}
        \lim_{X \to \infty} \frac{\# \{ p \leq X \mid p \text{ splits in } K, A^*_{\lambda}(p) \in [a,b] \}}{\# \{ p \leq X \mid p \text{ splits in } K \}} = \int_{a}^{b} B_3(t) \, dt
        .
    \end{equation}

    \item If $\lambda \in \{-4, 8, 64 \}$ and $-1 \leq a < b \leq 3$, then
    \begin{equation} \label{B4}
        \lim_{X \to \infty} \frac{\# \{ p \le X \mid p \text{ splits in } K, A^*_{\lambda}(p) \in [a, b]\}}{\# \{p \le X \mid p \text{ splits in } K \}} = \int_a^b B_4(t) \, dt
        .
    \end{equation}
\end{enumerate}
\end{theorem}

\subsection{Matrix Points} \label{sec: 2.3}
In order to prove Theorems \ref{thm: supersingular formula} and \ref{thm: k3 explicit formula}, we recall the formulas for matrix point counts on elliptic curves \eqref{eqn: aen} and $K3$ surfaces \eqref{eqn: alambdan}. To find an explicit matrix point count formula when $p$ is supersingular, we use more convenient versions of these formulas, as stated in the proofs of \cite[Theorem 1.1]{huangonosaad} and \cite[Theorem 1.4]{huangonosaad}. These formulas involve the eigenvalues of Frobenius $\alpha$ and $\overline{\alpha}$, which satisfy $\alpha + \overline{\alpha} = a_E(p)$ and $\alpha \overline{\alpha} = p$. More precisely, if $E$ is an elliptic curve, then we have
\begin{equation} \label{eq: matrix pt with eigenvalue}
N_{n}(E, \mathbb{F}_q) = (-1)^n p^{\frac{n(n-1)}{2}}(p;p)_n \sum_{\substack{r + s + u = n \\ r, s, u \ge 0}} \frac{(-1)^u \alpha^{r - s}p^{s + \frac{u(u + 1)}{2}}}{(p ; p)_r(p; p)_s(p; p)_u}
,
\end{equation}
where $\alpha$ is a Frobenius eigenvalue for $E$. For the $K3$ surfaces $X_\lambda$, we have
\begin{equation} \label{eq: K3 matrix pts}
\small
    N_{n}(X_{\lambda}, \mathbb{F}_p) = \sum_{r=0}^n b_{n-r} \hspace{-10pt} \sum_{\substack{\lambda_1 ,\dots, \lambda_6 \\ \lv\lambda_1\rv + \dots + \lv\lambda_6\rv = r}} \hspace{-10pt} (-1)^{l(\lambda_1) + \dots + l(\lambda_6)} \frac{p^{\frac{\sum n(\lambda_i , j)(n(\lambda_i , j) - 1)}{2}+2l(\lambda_3) + l(\lambda_4)}}{\prod (p;p)_{n(\lambda_i , j)}} \gamma^{l(\lambda_4) + l(\lambda_5) + l(\lambda_6)} \alpha^{2l(\lambda_5)} \overline{\alpha}^{2l(\lambda_6)}
    ,
\end{equation}
with $\lambda_i$, $n(\lambda_i, j)$, $b_r$, and $\gamma$ as in Theorem \ref{thm: k3 explicit formula} and where $\alpha$ and $ \overline{\alpha}$ are the eigenvalues of Frobenius for the associated Clausen elliptic curve $E_{\frac{-\lambda}{\lambda + 1}}^{\text{CL}}$. 

To obtain results about the distribution of $a_{E, n}^*(p)$, we observe that the difference between $a_{E, n}^*(p)$ and $a_E^*(p)$ vanishes as $p$ approaches infinity. In order to show this, we first recall  \cite[Theorem 1.1]{huangonosaad} that the polynomials $P(n, k)_p$ in \eqref{eqn: aen} are given by
\begin{equation} \label{eqn: P(n,k)}
    P(n, k)_p \coloneqq (-1)^k p^{n(n-k) + \frac{k(k+1)}{2}} \sum\limits_{s=0}^{\floor{\frac{n-k}{2}}} p^{2ks - 2ns + 2s^2} \frac{(p;p)_n}{(p;p)_s(p;p)_{s+k}(p;p)_{n-2s-k}}
    .
\end{equation}
We denote the difference between $a_{E,n}^*(p)$ and $a_E^*(p)$ by
\begin{equation} \label{equation: a_E, n}
    R_n^*(p) \coloneqq a_{E,n}^*(p) - a_E^*(p).
\end{equation}
We find an explicit bound for $R_n^*(p)$ in Lemma \ref{lemma: explicit constant}. 

We similarly observe that the difference between $A_{\lambda, n}^*(p)$ and $A^*_{\lambda}(p)$ vanishes as $p$ approaches infinity by combining \cite[Corollary 1.5]{huangonosaad} and definitions of $A_{\lambda}^*(p)$ from \cite[Section 1]{chenshenthorner}. More precisely, we have
\begin{equation} \label{equation: normalized K3 matrix points}
A_{\lambda, n}^*(p) = A^*_{\lambda}(p) + O_n(p^{-1}).
\end{equation}

These bounds allow us to determine the Sato--Tate type distributions of the matrix point count error for elliptic curves and $K3$ surfaces using information from the analogous finite field cases.

\section{Explicit Bounds}\label{section: 3}


To explicitly bound the difference $R_n^*(p)$ between $a_{E,n}^*(p)$ and $a_E^*(p)$, we first bound several combinatorial expressions involving the $q$-Pochhammer symbol.

\begin{lemma} \label{lemma: pochhammer simplify + bound}
    If $p\geq 3$ and $n\geq 1$, then 
    $$
    (p;p)_n=(-1)^n p^{\frac{n(n+1)}{2}} (p^{-1};p^{-1})_n
    $$
    and 
    $$
    0.5601 < \lp p^{-1}; p^{-1} \rp_n < 1
    .
    $$
\end{lemma}

\begin{proof}
    Expressing the $q$-Pochhamer symbols explicitly, we have
    $$
    (p;p)_n = \prod_{i=1}^n \lp 1-p^i \rp = \prod_{i=1}^n \lp -p^i \rp \lp 1-p^{-i} \rp = (-1)^n p^\frac{n(n+1)}{2} \lp p^{-1};p^{-1} \rp_n
    .
    $$
    
    To bound $(p^{-1};p^{-1})_n$, note that
    \begin{equation*}
        0.5601 < (3^{-1};3^{-1})_\infty \leq (p^{-1};p^{-1})_\infty < (p^{-1};p^{-1})_n < 1.\qedhere
    \end{equation*}
\end{proof}

To bound $R_n^*(p)$, it is convenient to divide $R_n^*(p)$ into three parts:
\begin{equation} \label{eqn: R_n(p) three parts}
    R_n^*(p) = Q_n^*(p) + S_n^*(p) - T_n^*(p),
\end{equation}
where
\begin{align*}
    Q_n^*(p) &\coloneqq a_E(p) \cdot p^{-n+\frac{1}{2}} \frac{(p; p)_n}{(p; p)_1(p;p)_{n-1}} - a_E(p) \cdot p^{-\frac{1}{2}},
    \\ S_n^*(p) &\coloneqq a_E(p) \cdot p^{-n+\frac{1}{2}} \sum_{s=1}^{\floor{\frac{n-1}{2}}} p^{2s - 2ns + 2s^2} \frac{(p;p)_n}{(p;p)_s(p;p)_{s+1}(p;p)_{n-2s-1}},
    \\ T_n^*(p) &\coloneqq p^{-n^2 + \frac{1}{2}} \sum_{k=2}^n a_E(p^k) \cdot (-1)^k p^{n(n-k) + \frac{k(k-1)}{2}} \sum_{s=0}^{\floor{\frac{n-k}{2}}} p^{2ks - 2ns + 2s^2} \frac{(p;p)_n}{(p;p)_s(p;p)_{s+k}(p;p)_{n-2s-k}}
    .
\end{align*}

The following three lemmas bound each of these terms.

\begin{lemma}\label{lemma: k = 1, s = 0 case}
If $n \geq 1$ and $p \geq 3$ is a prime, then
$$
\lv Q_n^*(p) \rv \leq 3p^{-1}
.
$$
\end{lemma}

\begin{proof}
    Using Lemma \ref{lemma: pochhammer simplify + bound} and \eqref{eqn: Hasse bound}, we get 
    \begin{align*}
        \lv Q_n^*(p) \rv 
        &= \lv a_E(p) \cdot p^{-n+\frac{1}{2}} \frac{(p; p)_n}{(p; p)_1(p;p)_{n-1}} - a_E(p) \cdot p^{-\frac{1}{2}} \rv
        \\ & \leq 2p^{\frac{1}{2}} \cdot \frac{p^{-\frac{3}{2}}}{1-p^{-1}}
        .
    \end{align*}
    Bounding the denominator using $p\geq 3$, we get
    \begin{equation}
        \lv Q_n^*(p) \rv \leq 3p^{-1}
        .\qedhere
    \end{equation}
\end{proof}

\begin{lemma}\label{lemma: k = 1, s > 0 case}
If $n \geq 1$ and $p \geq 3$ is a prime, then 
$$\lv S_n^*(p) \rv \leq 11.5247p^{-2}.$$

\end{lemma}

\begin{proof}
    Simplifying and using Lemma \ref{lemma: pochhammer simplify + bound} and \eqref{eqn: Hasse bound}, we get
    \begin{align*}
        \lv S_n^*(p) \rv
        &= \lv a_E(p) \cdot p^{-n+\frac{1}{2}} \sum_{s=1}^{\floor{\frac{n-1}{2}}} p^{2s - 2ns + 2s^2} \frac{(p;p)_n}{(p;p)_s(p;p)_{s+1}(p;p)_{n-2s-1}} \rv
        \\ &\leq 2 \sum_{s=1}^{\lfloor \frac{n-1}{2} \rfloor} p^{-s-s^2} \frac{1}{0.5601^3}.
    \end{align*}
Comparing the sum to a geometric series using $-s-s^2 \leq -4s + 2$, and then using $p\geq 3$, we have
    \begin{equation*}
        \lv S_n^*(p) \rv \leq 11.5247p^{-2}.\qedhere
    \end{equation*}
\end{proof}

\begin{lemma}\label{lemma: k > 1 case}
If $n \geq 1$ and $p \geq 3$ is a prime, then 
$$
\lv T_n^*(p) \rv \leq 12.6305p^{-\frac{3}{2}}
.
$$
\end{lemma}

\begin{proof}
    This follows exactly as in Lemma \ref{lemma: k = 1, s > 0 case}.
\end{proof}

Combining the previous three lemmas and using $p\geq 3$, we obtain the bound of $R_n^*(p)$.
\begin{lemma} \label{lemma: explicit constant}

If $n \geq 1$ and $p \geq 3$ prime, then
\[
\pushQED{\qed}
\lv R_n^*(p) \rv \leq 14.1339p^{-1}
.\qedhere
\popQED
\]
\end{lemma}

Lemma \ref{lemma: explicit constant} implies that $a_{E,n}^*(p)$ and $a_{E}^*(p)$ are close for large primes. To bound the error for small primes, we use an approximation of the prime counting function \cite{primenumbertheorem}. More precisely, if $X \geq 17$, then
\begin{equation} \label{eqn: prime number theorem}
    \frac{X}{\log X} < \# \{ p \leq X \} < 1.2551 \frac{X}{\log X}
    .
\end{equation}

\begin{remark}
It follows from \eqref{eqn: prime number theorem} that given any $A, X \ge 17$, we have
    \begin{equation} \label{eqn: PNT}
        \frac{\# \{p \le A \}}{\# \{ p \le X \}} \le 1.2551 \frac{A \log X}{X \log A}.
    \end{equation}
\end{remark}

Finally, we want to absorb the bounds for $R_n^*(p)$ and the prime counting function into the explicit bound in \eqref{eqn: ST bound}. To this end, note that since $N$ is the conductor of an elliptic curve, we have $N \geq 11$. It follows that
    \begin{equation} \label{eqn: X^-1/2 bound}
        X^{-\frac{1}{2}} \leq 0.2429\frac{\log(N\log X)}{\sqrt{\log X}}
        .
    \end{equation}

\section{Proofs of Theorems}\label{section: 4}

Here we use the results from the previous sections to prove Theorems \ref{thm: supersingular formula}, \ref{thm: k3 explicit formula}, \ref{thm: ineffective theorems}, \ref{thm: effective theorems}, and \ref{thm: K3 bounding theorems}.

\subsection{Proof of Theorems \ref{thm: supersingular formula} and \ref{thm: k3 explicit formula}}

Here we prove Theorems \ref{thm: supersingular formula} and \ref{thm: k3 explicit formula}.

\begin{proof}[Proof of Theorem \ref{thm: supersingular formula}]
If $p$ is supersingular for $E$, then $a_E^*(p) = 0$. This implies that the Frobenius eigenvalues are given by $\alpha = \sqrt{p}i$ and $\overline{\alpha} = - \sqrt{p}i$. Substituting into \eqref{eq: matrix pt with eigenvalue}, we get
\begin{align*}
N_{n}(E, \mathbb{F}_p) & = (-1)^n p^{\frac{n(n-1)}{2}}(p;p)_n \sum_{\substack{r + s + u = n \\ r, s, u \ge 0 }} \frac{(-1)^u (i)^{r - s} p^{\frac{r - s}{2}}p^{s + \frac{u(u + 1)}{2}}}{(p ; p)_r(p; p)_s(p; p)_u}
\\ & = (-1)^n p^{\frac{n^2}{2}} \sum_{\substack{r + s + u = n \\ r, s, u \ge 0}} (-1)^u (i)^{r - s}p^{\frac{u^2}{2}} \binom{n}{r, s, u}_p
.
\end{align*}
Pairs of terms in our summand $(r_1, s_1, u)$ and $(r_2, s_2, u)$ with $s_1 = r_2$ and $r_1 = s_2$ will differ by a sign when $r-s$ is odd. Therefore, we have
\begin{equation*}
N_{n}(E, \mathbb{F}_p) = p^{\frac{n^2}{2}} \sum_{\substack{r + s + u = n \\ r, s, u \ge 0 \\ r \equiv s \, \text{mod} \, 2}} (-1)^{\frac{r - s}{2}} p^{\frac{u^2}{2}} \binom{n}{r, s, u}_p. \qedhere
\end{equation*}
\end{proof}

\begin{proof}[Proof of Theorem \ref{thm: k3 explicit formula}]
Let $\alpha$ and $\overline{\alpha}$ be the Frobenius eigenvalues of the associated Clausen elliptic curve of $X_\lambda$.
If $p$ is a supersingular prime for this curve, then $\alpha^2 = \overline{\alpha}^2 = -p$. Substituting this into \eqref{eq: K3 matrix pts} and then rearranging powers of $p$, the statement of the theorem follows.
\end{proof}

\subsection{Proof of Theorems \ref{thm: ineffective theorems} and \ref{thm: K3 bounding theorems}}

Here we prove Theorems \ref{thm: ineffective theorems} and \ref{thm: K3 bounding theorems}.

\begin{proof}[Proof of Theorem \ref{thm: ineffective theorems} (i)]
    Fix $\varepsilon > 0$. By \eqref{equation: a_E, n} and Lemma \ref{lemma: explicit constant}, we have
    $$
    a_{E,n}^*(p) = a_E^*(p) + R_n^*(p)
    $$
    where $\lv R_n^*(p) \rv \leq 14.15p^{-1}$. Choosing $A > \frac{14.15}{\varepsilon}$, we have $\lv R_n^*(p) \rv < \varepsilon$ for every $p \geq A$. We note that if $p \geq A$, then $[a-R_n^*(p),b-R_n^*(p)] \subseteq [a-\varepsilon, b+\varepsilon]$. Therefore, 
    \begin{equation*}
        \frac{\#\{A \le p \le X \mid a_{E,n}^*(p) \in [a, b]\}}{\{p\le X\}} 
        \le \frac{\#\{A \le p \le X \mid a_E^*(p) \in [a - \varepsilon, b + \varepsilon]\}}{\{p\le X\}}
        .
    \end{equation*}

   It follows that
    \begin{equation*}
        \limsup_{X \to \infty} \frac{\#\{A \le p \le X \mid a_{E,n}^*(p) \in [a, b]\}}{\{p\le X\}} 
        \le \limsup_{X \to \infty} \frac{\#\{A \le p \le X \mid a_E^*(p) \in [a - \varepsilon, b + \varepsilon]\}}{\{p\le X\}} 
        .
    \end{equation*}

   Applying \eqref{eq: og Sato--Tate conjecture} to the right-hand side, noting that $A$ is finite, and bounding with $f_{ST} (t) \leq \frac{1}{\pi}$, we conclude
    \begin{align*}
        \limsup_{X \to \infty} \frac{\#\{p \le X \mid a_{E,n}^*(p) \in [a, b]\}}{\{p\le X\}} 
        & \le \int_{a-\varepsilon}^{b+\varepsilon} f_{ST}(t) \, dt
        \\ &\le \int_{a}^{b} f_{ST}(t) \, dt + \frac{2\varepsilon}{\pi}.
    \end{align*}

   Analogously, using $[a+\varepsilon,b-\varepsilon] \subseteq [a-R_n^*(p),b-R_n^*(p)]$, we have that
   \begin{equation*}
       \liminf_{X \to \infty} \frac{\#\{p \le X \mid a_{E,n}^*(p) \in [a, b]\}}{\{p\le X\}} \geq \int_{a}^{b} f_{ST}(t) \, dt - \frac{2\varepsilon}{\pi}
       .
   \end{equation*}
   
   Since $\varepsilon > 0$ is arbitrary, the proof is complete.
\end{proof}

\begin{proof}[Proof of Theorem \ref{thm: ineffective theorems} (ii)] The proof for $E$ with CM follows analogously, where we replace $f_{ST}(t)$ with the distribution in \eqref{eqn: cm limiting distribution} and consider only the primes which split in the CM field of $E$. We point out that even though $\frac{1}{\pi} \cdot \frac{1}{\sqrt{4-t^2}}$ is unbounded, the argument remains valid since its integral over an interval of length $\varepsilon$ converges to $0$ as $\varepsilon$ approaches $0$.
\end{proof}

\begin{proof}
[Proof of Theorem \ref{thm: K3 bounding theorems}]
From \eqref{equation: normalized K3 matrix points}, the difference between $A_{\lambda,n}^*(p)$ and $A_\lambda^*(p)$ vanishes as $p$ approaches infinity. The proof of Theorem \ref{thm: K3 bounding theorems} thus follows analogously to the proof of Theorem \ref{thm: ineffective theorems} (i), where Theorem \ref{thm: 2.2} plays the role of \eqref{eq: og Sato--Tate conjecture} and the bounding of the integrals is analogous to the proof of Theorem \ref{thm: ineffective theorems} (ii).
\end{proof}
\pagebreak
\subsection{Proof of Theorem \ref{thm: effective theorems}}

Here we prove Theorem \ref{thm: effective theorems}.

\begin{proof} [Proof of Theorem \ref{thm: effective theorems}]
We write $a_{E,n}^*(p) = a_E^*(p) + R_n^*(p)$. Using Lemma \ref{lemma: explicit constant}, we have that $$[a-R_n^*(p), b-R_n^*(p)] \subseteq [a - Cp^{-1}, b + Cp^{-1}],$$ where $C = 14.1339$. This implies that
\begin{multline*}
\frac{\#\{p \le X \mid a_{E, n}^*(p) \in [a, b]\}}{\#\{p \le X\}} - \int_{a}^{b} f_{ST}(t) \, dt
\\ \le \frac{\#\{p \le X \mid a_{E}^*(p) \in [a - Cp^{-1}, b + Cp^{-1}]\}}{\#\{p \le X\}} - \int_{a}^{b} f_{ST}(t) \, dt
.
\end{multline*}
In order to bound the right-hand side, we consider small and large primes separately. To this end, for any $A$ such that $17 \leq A \leq X$, we have
\begin{align*}
& \frac{\#\{p \le X \mid a_{E, n}^*(p) \in [a, b]\}}{\#\{p \le X\}} - \int_{a}^{b} f_{ST}(t) \, dt 
\\ & \leq \frac{\#\{p \le A \mid a_{E}^*(p) \in [a - Cp^{-1}, b+Cp^{-1}]\}}{\#\{p \le X\}} + \frac{\#\{A \le p \le X \mid a_{E}^*(p) \in [a - Cp^{-1}, b+Cp^{-1}]\}}{\#\{p \le X\}} 
\\ & \qquad - \int_{a}^{b} f_{ST}(t) \, dt 
\\ & \leq \frac{\#\{p \le A\}}{\#\{p \le X\}} + \frac{\#\{A \le p \le X \mid a_{E}^*(p) \in [a - CA^{-1}, b+CA^{-1}]\}}{\#\{p \le X\}} 
\\ & \qquad - \lp \int_{a-CA^{-1}}^{b+CA^{-1}} f_{ST}(t) \ dt - \int_{a-CA^{-1}}^{a} f_{ST}(t) \, dt - \int_{b}^{b+CA^{-1}} f_{ST}(t) \, dt \rp
.
\end{align*}
Using \eqref{eqn: PNT}, the explicit bound from \eqref{eqn: ST bound}, and $f_{ST} \leq \frac{1}{\pi}$, we have
\begin{align*}
& \frac{\#\{p \le X \mid a_{E, n}^*(p) \in [a, b]\}}{\#\{p \le X\}} - \int_{a}^{b} f_{ST}(t) \, dt
\\ & \le 1.2551 \frac{A \log X}{X \log A} + \lp \frac{\#\{p \le X \mid a_{E}^*(p) \in [a - CA^{-1}, b+CA^{-1}]\}}{\#\{p \le X\}} -  \int_{a-CA^{-1}}^{b+CA^{-1}} f_{ST}(t) \, dt  \rp + \frac{2C}{\pi}A^{-1}
\\ & \le 1.2551 \frac{A \log X}{X \log A} + 58.1\frac{\log \lp N \log X \rp}{\sqrt{\log X}} + \frac{2C}{\pi}A^{-1}
.
\end{align*}
Choosing $A = X^{\frac{1}{2}}$ and applying \eqref{eqn: X^-1/2 bound}, we get
\begin{align*}
& \frac{\#\{p \le X \mid a_{E, n}^*(p) \in [a, b]\}}{\#\{p \le X\}} - \int_{a}^{b} f_{ST}(t) \, dt 
\\ & \le \lp 2 \cdot 1.2551 + \frac{2 \cdot 14.1339}{\pi} \rp X^{-\frac{1}{2}} + 58.1 \lp \frac{\log \lp N \log X \rp}{\sqrt{\log X}} \rp
\\ & \leq 58.44 \frac{\log \lp N \log X \rp}{\sqrt{\log X}}
.
\end{align*}
Since the lower bound follows analogously, the proof is complete.
\end{proof}

\section{Numerical Examples}\label{section: 5}
We provide histograms of the normalized $2 \times 2$ matrix point count errors for primes $p \le 10^6$. The limiting distributions are overlaid.

\begin{figure}[H]
    \centering
    \includegraphics[width = 10cm]{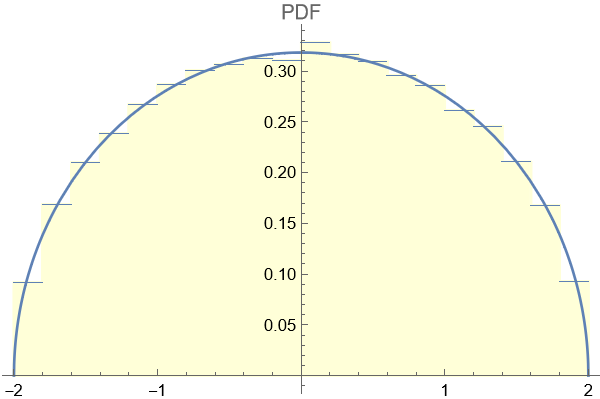}
    \caption*{\textsc{Figure 1 (Semicircular).} Histogram of $a_{E,2}^*(p)$ for \\ $E \colon y^2 = x^3 - 432x + 8208$.}
\end{figure}

\begin{figure}[H]
    \centering
    \includegraphics[width = 10cm]{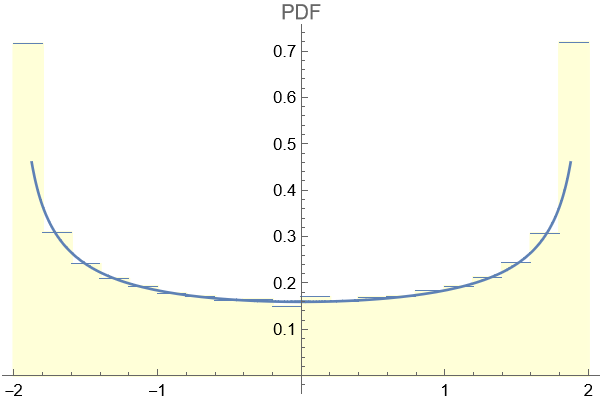}
    \caption* {\textsc{Figure 2.} Histogram of $a_{E,2}^*(p)$ for \\ $E \colon y^2 = x^3 + x$.}
\end{figure}

\begin{figure}[H]
    \centering
    \includegraphics[]{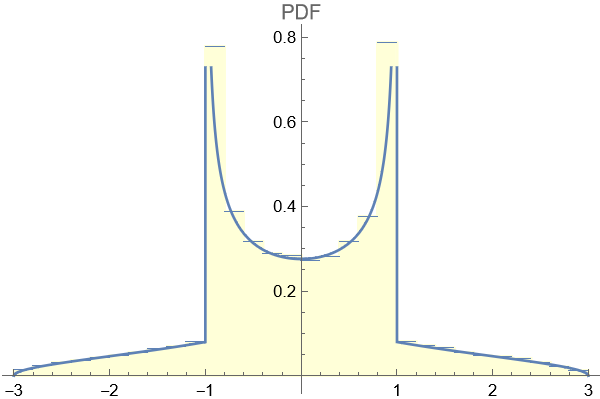}
    \caption*{\textsc{Figure 3 (Original Batman).} Histogram of $A_{\lambda, 2}^*(p)$ for $\lambda = 5.$}
\end{figure}

\begin{figure}[H]
    \centering
    \includegraphics[]{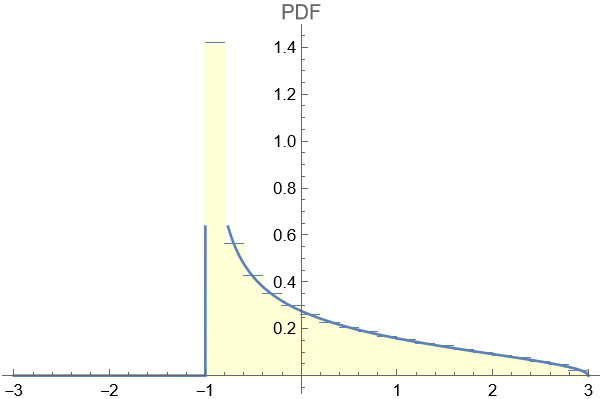}
    \caption*{\textsc{Figure 4 (Half-Batman).} Histogram of $A_{\lambda,2}^*(p)$ for $\lambda = 3.$}
\end{figure}

\begin{figure}[H]
    \centering
    \includegraphics[]{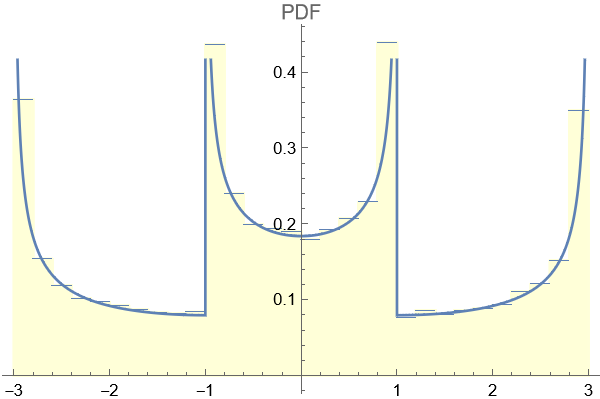}
    \caption*{\textsc{Figure 5 (Flying Batman).} Histogram of $A_{\lambda,2}^*(p)$ for $\lambda = 1$.}
\end{figure}

\begin{figure}[H]
    \centering
    \includegraphics[]{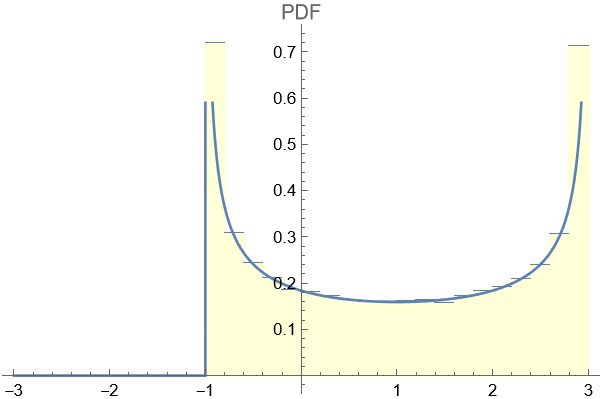}
    \caption*{\textsc{Figure 6 (Half-Flying Batman).} Histogram of $A_{\lambda, 2}^*(p)$ for  $\lambda = -4$.}
\end{figure}

\bibliography{main}

\begin{thebibliography}{10}

\bibitem{aopk3}
S.~Ahlgren, K.~Ono, and D.~Penniston.
\newblock Zeta functions of an infinite family of {$K3$} surfaces.
\newblock {\em Amer. J. Math.}, 124(2):353--368, 2002.

\bibitem{taylor2}
T.~Barnet-Lamb, D.~Geraghty, M.~Harris, and R.~Taylor.
\newblock A family of {C}alabi-{Y}au varieties and potential automorphy {II}.
\newblock {\em Publ. Res. Inst. Math. Sci.}, 47(1):29--98, 2011.

\bibitem{chenshenthorner}
Q.~Chen, E.~Shen, and J.~Thorner.
\newblock Effective {S}ato--{T}ate distributions for surfaces arising from
  products of elliptic curves.
\newblock Pending.

\bibitem{sato1}
L.~Clozel, M.~Harris, and R.~Taylor.
\newblock Automorphy for some $\ell$-adic lifts of automorphic mod $\ell$
  {G}alois representations.
\newblock {\em Publ. Math. Inst. Hautes \'{E}tudes Sci.}, 108:1--181, 2008.
\newblock With Appendix A, summarizing unpublished work of Russ Mann, and
  Appendix B by Marie-France Vign\'{e}ras.

\bibitem{elkies1987}
N.~D. Elkies.
\newblock The existence of infinitely many supersingular primes for every
  elliptic curve over $\mathbb{Q}$.
\newblock {\em Invent. Math.}, 89(3):561--567, 1987.

\bibitem{taylor1}
M.~Harris, N.~Shepherd-Barron, and R.~Taylor.
\newblock A family of {C}alabi-{Y}au varieties and potential automorphy.
\newblock {\em Ann. of Math. (2)}, 171(2):779--813, 2010.

\bibitem{hasse1936}
H.~Hasse.
\newblock Zur {T}heorie der abstrakten elliptischen {F}unktionenk\"{o}rper
  {III}. {D}ie {S}truktur des {M}eromorphismenrings. {D}ie {R}iemannsche
  {V}ermutung.
\newblock {\em J. Reine Angew. Math.}, 175:193--208, 1936.

\bibitem{hecke}
E.~Hecke.
\newblock Eine neue {A}rt von {Z}etafunktionen und ihre {B}eziehungen zur
  {V}erteilung der {P}rimzahlen.
\newblock {\em Math. Z.}, 6(1-2):11--51, 1920.

\bibitem{hoey2022unconditional}
A.~Hoey, J.~Iskander, S.~Jin, and F.~Trejos~Su\'{a}rez.
\newblock An unconditional explicit bound on the error term in the
  {S}ato-{T}ate conjecture.
\newblock {\em Q. J. Math.}, 73(4):1189--1225, 2022.

\bibitem{huang2022mutually}
Y.~Huang.
\newblock Mutually annihilating matrices, and a {C}ohen--{L}enstra series for
  the nodal singularity.
\newblock {\em J. Algebra}, 619:26--50, 2023.

\bibitem{huangonosaad}
Y.~Huang, K.~Ono, and H.~Saad.
\newblock Counting matrix points on certain varieties over finite fields.
\newblock {\em Contemp. Math., Amer. Math. Soc.}, accepted for publication.
\newblock \url{https://arxiv.org/abs/2302.04830}.

\bibitem{onosaadsaikia}
K.~Ono, H.~Saad, and N.~Saikia.
\newblock Distribution of values of {G}aussian hypergeometric functions.
\newblock {\em Pure Appl. Math. Q.}, 19(1):371--407, 2023.

\bibitem{primenumbertheorem}
J.~B. Rosser and L.~Schoenfeld.
\newblock Approximate formulas for some functions of prime numbers.
\newblock {\em Illinois J. Math.}, 6:64--94, 1962.

\bibitem{silvermanbook}
J.~H. Silverman.
\newblock {\em The Arithmetic of Elliptic Curves}.
\newblock Springer New York, NY, 2009.

\bibitem{sato2}
R.~Taylor.
\newblock Automorphy for some $\ell$-adic lifts of automorphic mod $\ell$
  {G}alois representations. {II}.
\newblock {\em Publ. Math. Inst. Hautes \'{E}tudes Sci.}, 108:183--239, 2008.

\end{thebibliography}

\Addresses

\end{document}